\documentclass[12pt]{article}
\usepackage{amsmath,amssymb}
\usepackage{amsthm}

\def\PC{\mathcal{P}}

\def\BC{\mathcal{B}}

\def\E{\mathbf{E}}
\def\N{\mathbf{N}}
\def\P{\mathbf{P}}
\def\R{\mathbf{R}}
\def\Z{\mathbf{Z}}

\def\1{\mathbf{1}}

\def\al{\alpha}

\def\ep{\epsilon}
\def\de{\delta}
\def\ga{\gamma}

\newtheorem{prop}{Proposition}[section]
\newtheorem{theorem}{Theorem}[section]

\newcommand{\la}{\lambda}

\newcommand{\om}{\omega}

\newcommand{\Si}{\Sigma}
\newcommand{\Om}{\Omega}

\begin{document}
\title{On some models of many agent systems with competition and cooperation
%\thanks{http://arxiv.org/abs/1105.3053}
%\thanks{Supported by the AFOSR grant FA9550-09-1-0664 'Nonlinear Markov control processes and games'}
}
\author{Vassili N. Kolokoltsov\thanks{Department of Statistics, University of Warwick,
 Coventry CV4 7AL UK,
  Email: v.kolokoltsov@warwick.ac.uk} and Oleg A. Malafeyev\thanks{St Petersburg University, Russia}}
\maketitle

\begin{abstract}
We review and develop a selection of models of systems with competition and cooperation, with origins in economics,
where deep insights can be obtained by the mathematical methods of
game theory. Some of these models were touched upon in authors' book 'Understanding Game Theory', World Scientific 2010, where also the necessary background on games can be found.
\end{abstract}


\vspace{5mm}

We review and develop below some models of competition and cooperation, with origins in economics. These models
are linked only ideologically and methodologically, which allows one to read all sections almost independently. 

\section{Territorial price building}

Here we present a variation of the classical Cournot model of price building, where selling and production sites are spatially distributed.

Suppose there is a market that include selling sites
 $M_1,\ldots , M_m$, products $1,\ldots , K$ and production sites $1,...,L$.

There are $N$ agents that buy the products at the production sites
 and sell them in $M_1,\ldots , M_m$.

Assume that at the each production site $l$ there is a large stock
of all products available at the price $p_{kl}$ for the unit of the
product $k$, and that the transportation cost to $M_i$ of a unit of
the product $k$ is $\xi_{ikl}$.

 Let $Y^n_{ikl}$ denote the amount of the product $k$, which the
 agent $n$ brought to $M_i$ from the site $l$. Then the total amount of $k$th product in $M_i$
 is
 \[
 Y_{ik}=\sum\limits_{n=1}^NY^n_{ik}, \quad Y^n_{ik}=\sum_{l=1}^L Y^n_{ikl}.
 \]
  The selling price of the
 product $k$ in $M_i$ is the function of the total supply of
the product. Following the standard Cournot model assume that this price is given by the formula

\[
R_{ik}(Y_{ik})=(1-Y_{ik}/\alpha_{ik})\beta_{ik}, \quad Y_{ik} \in
[0, \alpha_{ik}],
\]
with some positive constants $\alpha_{ik}, \beta_{ik}$.

 Then the total income of the player $n$ can be calculated by the formula
\begin{equation}
 \label{ter0}
H_n=\sum_{i=1}^m\sum_{k=1}^K H_n(ik)
\end{equation}
with
 \begin{equation}
 \label{ter1}
H_n(ik)=Y^n_{ik}
  \left(1-\frac{Y_{ik}}{\alpha_{ik}}\right)\beta_{ik}
  -\sum_{l=1}^L Y^n_{ikl}(\xi_{ikl}+p_{kl}), \quad Y_{ik} \in
[0, \alpha_{ik}].
 \end{equation}
 Thus we have defined a symmetric $N$-person game $\Gamma^N$ with the payoffs given by \eqref{ter0}
 and with the
 strategies of each player $n$ being the arrays
 $(Y^n_{ikl})$ with positive entries bounded by $\alpha_{ik}$.

For simplicity we reduce our attention to the case of 2 players.

\begin{prop}
 \label{terr1}
The game $\Gamma^2$ has a symmetric Nash equilibrium. If for each
pair $(i,k)$ there exists a unique $l=q(i,k)$ where the minimum of
$\xi_{ikl}+p_{kl}$ is attained, then this equilibrium is unique and
the corresponding strategies of each player are given by
\begin{equation}
 \label{ter2}
 \hat Y_{ik}=\frac{1}{3} \delta_{q(i,k)}^l\alpha_{ik} \left(1-\frac{\xi_{ikl}+p_{kl}} {
 \beta_{ik}}\right),
 \end{equation}
 where $\delta_m^l$ denotes the usual Dirac symbol that equals 1 for
 $l=m$ and zero otherwise.
\end{prop}

\begin{proof} Clearly looking for a best reply one has to maximize each
$H_n(ik)$ separately for each pair $(ik)$. It is also clear from
\eqref{ter1} that to have a best reply one has to bring the product
only from those sites that minimize $\xi_{ikl}+p_{kl}$. Hence given
$(Y_{ik}^1)$ to find the best reply of the second player one has to
find the maximum of the functions
\[
\bar H_2^{ik}(Y_{ikq(i,k)}^2)= Y^n_{ikq(i,k)})
   \left(1-\frac{Y^1_{ik}+Y^2_{ikq(i,k)}}{\alpha_{ik}}\right)\beta_{ik}
  -Y^n_{ikq(i,k)}(\xi_{ikq(i,k)}+p_{kq(i,k)})
  \]
 for each pair $(ik)$.
 Differentiating one gets
\[
\frac{d\bar H_2^{ik}(z)}{dz}
 =\left(1-\frac{Y^1_{ik}+z}{\alpha_{ik}}\right)\beta_{ik}
  -(\xi_{ikq(i,k)}+p_{kq(i,k)})-z\frac{\beta_{ik}}{\alpha_{ik}}=0
  \]
  for the best reply $z=Y^2_{ikq(i,k)}$,
  yielding
\begin{equation}
 \label{ter3}
 Y^2_{ikq(i,k)}=\frac{1}{2} \alpha_{ik} \left(1-\frac{Y^1_{ik}}{\alpha_{ik}}
 -\frac{\xi_{ikq(i,k)}+p_{kq(i,k)}} { \beta_{ik}}\right).
 \end{equation}
As the second derivatives of $\bar H_2^{ik}$ is negative this point
does really define a maximum (and not a minimum).
 By the symmetry one concludes that symmetric equilibrium strategy
 solves the equation
\begin{equation}
 \label{ter4}
 \hat Y_{ik}=\frac{1}{2} \alpha_{ik}\left(1-\frac {\hat Y_{ik}}{\alpha_{ik}}
 -\frac{\xi_{ikq(i,k)}+p_{kq(i,k)}}{\beta_{ik}}\right)
 \end{equation}
 and thus equals
 \eqref{ter2} as required.
 \end{proof}

\section{Bi-matrix inspection games}

In this and the next sections we develop some extensions of the classical inspection game.
For other similar models see \cite{AvKi},
 \cite{AvCa}, \cite{FeM}, \cite{KoMabook} and references therein.
 
The game is carried out between a trespasser (player I) and an
inspector(player II) in $n$-periods (or steps). We shall analyze a
'stop to abuses' version of this game when it lasts only till the
first time trespassing is discovered by the inspector.

Player I has 2 pure strategies: to break (violate) the law (B) or to
refrain from it (R). Player II has also 2 pure strategies: to check
the actions of player I (C) or to have a rest (R). If player I
chooses (R) he gets the legal income  $r>0$. If he chooses (B), he
obtains additionally the illegal surplus $s>0$. However, if his
illegal action is discovered by player II, the player I pays the
fine $f>0$ and the game is over.

If player II chooses (C) he spends the amount $c>0$ on this
procedure and can discover the trespassing of player I with the
probability $p(\overline{p}=1-p)$. If player I breaks the law
and this action is not discovered, the inspector loses the amount
$l>0$.

Consequently one step of this game can be described by the table

\begin{center}
\begin{tabular}{cc}
& Inspector \\
 Trespasser &
\begin{tabular}{|c|c|c|}
\hline
& Check (C) & Rest (R) \\
\hline
Break (B) & $-pf+\overline{p}(r+s), -(c+\overline{p}l)$ & $r+s, -l$ \\
\hline
Refrain (R) & $r,-c$ & $r,0$ \\
\hline
\end{tabular}
\end{tabular}
\end{center}

or shortly by the matrix

\begin{equation}
\label{Inspect3}
\left(
\begin{array}{cc}
  -pf+\overline{p}(r+s), -(c+\overline{p}l) & \qquad r+s, -l  \\
  \quad r, -c \quad & \quad r, 0 \quad \end{array}
  \right)
\end{equation}

It is natural to assume, that $c<pl$, so that the pair (B,R) is
not a Nash equilibrium (otherwise the inspector has no reasons at
all to conduct checks).

Let us say that a bi-matrix game {\it has a value}, if payoffs in
any of the existing Nash equilibrium are the same, the corresponding
pair of payoffs being called the {\it value of the game}\index{value
of the game}.

\begin{prop}
\label{leinspect} Assume $f,r,s,c,l>0$, $c<pl$. Then the game with
the matrix \eqref{Inspect3} has a value $V=(u,v)$. Moreover, 1) if
 $ s< s_{1}=\frac{p}{\overline{p}}(f+r)$, the unique Nash equilibrium
 is given by the pair of mixed strategies $(x,\overline{x})$, $(y,\overline{y})$,
 where
 \[
 x=\frac{c}{pl} , y=\frac{s}{p(f+r+s)};
 \quad u=r,v=-\frac{c}{p}.
 \]
 2) if $s>s_{1}$, the unique Nash equilibrium
 is the profile (B,C) and
 \[u=-pf +\overline{p}(r+s),v=-(c+\overline{p}l);
 \]
 3) if $s=s_1$, then  $u=r,v=-(c+\overline{p}l)$ and the Nash
 equilibria are given by all pairs $(X,C)$, where $X$ is any (pure or mixed)
 strategy of the trespasser.
\end{prop}

\begin{proof} Clearly the only candidate for Nash equilibrium in pure
strategies is the pair $(B,C)$. This profile is an equilibrium if
and only if $s\ge s_{1}$. Formulae for $x$, $y$ follow from
 the standard expression for mixed strategy profiles in two-action two-player games. Other statements are checked by a straightforward inspection.
\end{proof}

Of course of greater interest is the analysis of a multi-step
version of this game. Let us consider the $n$-step game, where
during this time the player I can break the law at most $k$ times
and the player II can organize the check at most $m$ times. Assume
that after the end of each period (step), the result becomes known
to both players. Total payoff in $n$ steps equals the sum of payoffs
in each step. It is also assumed that all this information (rules of
the game) is available to both players.

Let us denote the game described above by $\Gamma_{k,m}(n)$.
 Let
$(u_{k,m}(n),v_{k,m}(n))$ be the value of this game. We then get the
following system of recurrent equations:
\[
(u_{k,m}(n),v_{k,m}(n))
\]
\begin{equation}
= {Val} \left(
\begin{array}{cc}
  -pf+\overline{p}(r+s+u_{k-1,m-1}(n-1)), & r+s+u_{k-1,m}(n-1),  \\
  -(c+\overline{p}l)+\overline{p}v_{k-1,m-1}(n-1) & -l+v_{k-1,m}(n-1)  \\
  r+u_{k,m-1}(n-1), -c+v_{k,m-1}(n-1) & r+u_{k,m}(n-1), v_{k,m}(n-1)
\end{array}
\right)
\label{Inspect4}
\end{equation}
(if all $\Gamma_{k,m}(n)$ have values, i.e. their equilibrium
payoffs are uniquely defined), with the boundary conditions ($m,n,k
\ge 0$):

\begin{equation}
(u_{0,m}(n),v_{0,m}(n))=(nr,0)  ; \label{Inspect5}
\end{equation}

\begin{equation}
(u_{k,0}(n),v_{k,0}(n))=(nr+ks,-kl)  ; \quad k\leq n,
\label{Inspect6}
\end{equation}
 reflecting the following considerations: if the trespasser is unable to break the law, the pair of
solutions (R,R) will be repeated over all periods; and if the
inspector is unable to check, the trespasser will commit the maximum
number of violations available.

Though $k \le n,m\ge n$, the form of the recurrent equations below
is slightly simplified if one allows all non-negative $k,m, n$
together with the agreement
\begin{equation}
(u_{k,m}(n),v_{k,m}(n))=(u_{k',m'}(n),v_{k',m'}(n)); k'=\min(k, n),
m'=\min(m, n) \label{Inspect7}
\end{equation}

Let us reduce our attention further to the game $\Gamma_{n,n}(n)$.
Let us write $U_{n}=u_{n,n}(n)$ and $V_{n}=v_{n,n}(n)$. Then
(\ref{Inspect4}) takes the form:
\begin{equation}
(U_{n},V_{n})= (U_{n-1},V_{n-1})+ {Val} M_{n},
\end{equation}

where
\begin{equation}
M_{n}= \left( \begin{array}{cc}
   \overline{p}(r+s)-p(f+U_{n-1}), -(c+\overline{p}l+pV_{n-1}) & \quad r+s, -l
   \\
   r, -c & \quad r, 0
\end{array}
 \right)
\end{equation}
\[
(n\geq 0; \quad U_{0}=V_{0}=0)
\]

For $n=1$ the game becomes the same as the game (\ref{Inspect3}),
and its solution is given by Proposition \ref{leinspect}.

Let us find the solution to the game $\Gamma_{2,2}(2)$. Plugging the
values $U_{1}$ and $V_{1}$ into $M_{2}$ yields
\[
M_{2}= \left( \begin{array}{cc}
   \overline{p}(r+s)-p(f+r), -\overline{p}l & r+s, -l  \\
   r, -c & r, 0
\end{array} \right) , \quad 0<s<s_{1};
\]
\[
M_{2}= \left( \begin{array}{cc}
   \overline{p}(-pf+\overline{p}(r+s)), -\overline{p}(c+\overline{p}l) & r+s, -l
   \\
   r, -c & r, 0
\end{array} \right) , \quad s>s_{1};
\]
\[
s_{2}=\frac{p}{\overline{p}}(f+r)+\frac{p}{\overline{p}^{2}}r.
\]

Direct calculations show that under the assumptions of Proposition
\ref{leinspect} the $M_{2}$ also has a value, and one can
distinguish three basic cases (equilibrium strategies are again
denoted by $(x,\overline{x})$, $(y,\overline{y})$):

   1) if $0<s<s_{1}$, then
   \[
   x=\frac{c}{pl+c} , y=\frac{s}{p(f+2r+s)} ;
   \quad U_{2}=2r,V_{2}=-\frac{c(2pl+c)}{p(pl+c)}
   \]

   2) if $s_{1}<s<s_{2}$, then
   \[
   x=\frac{c}{p(c+(1+\overline{p})l)},
   \quad  y=\frac{s}{p(\overline{p}f+(1+\overline{p})(r+s))} ;
   \]
   \[
   U_{2}=-pf+\overline{p}(r+s),
   \quad V_{2}=-(\frac{cl}{p(c+(1+\overline{p})l)}+c+\overline{p}l)
   \]

   3) if $s>s_{2}$, then
   \[
   U_{2}=(1+\overline{p})(-pf+\overline{p}(r+s)), \quad
   V_{2}=-(1+\overline{p})(c+\overline{p}l),
   \]
and the profile (B,C) is an equilibrium.

Analogously one can calculate the solutions for other $n>2$.

\section{Tax payer against a tax man}

This is a game between a tax payer (player I) and the tax police
(player II). Player I has 2 pure strategies: to hide part of the
taxes (H) or to pay them in full (P). Player II has also 2
strategies: to check player I (C) and to rest (R). Player I
gets the income $r$ if he pays the tax in full. If he chooses the
action (H), he gets the additional surplus $l$. But if he is caught
by player II, he has to pay the fine $f$.

In the profile (C,H) player II can discover the unlawful action
of player I with the probability $p(\overline{p}=1-p)$, so that $p$
can be called the efficiency of the police. Choosing (C), player
II spends $c$ on the checking procedure. Of course $l,r,f,c>0$.

Hence we defined a bi-matrix game given by the table

\begin{center}
\begin{tabular}{cc}
& Player II (Police) \\
  Player I &
\begin{tabular}{|c|c|c|}
\hline
& Check (C) & Rest (R) \\
\hline
Hide (H) & $r+\overline{p}l-pf, -c+pf-\overline{p}l$ & $r+l, -l$ \\
\hline
Pay (P) & $r,-c$ & $r,0$ \\
\hline
\end{tabular}
\end{tabular}
\end{center}

or shortly by the payoff matrix
\begin{equation}
\left( \begin{array}{cc}
  r+\overline{p}l-pf, -c+pf-\overline{p}l & r+l, -l  \\
  r, -c & r, 0
\end{array} \right) \label{Inspect11}
\end{equation}

 The candidates to the mixed equilibrium are the strategies
 $(\beta,\overline{\beta})$, $(\alpha,\overline{\alpha})$, where
\[
\alpha=\frac{a_{22}-a_{12}}{a_{11}-a_{12}-a_{21}+a_{22}}
=\frac{l}{p(l+f)}>0
\]
\[
\beta=\frac{b_{22}-b_{21}}{b_{11}-b_{12}-b_{21}+b_{22}}
=\frac{c}{p(l+f)}>0.
\]
In order to have these strategies well defined, it is necessary to
have $\alpha <1$ and $\beta <1$ respectively. By a direct inspection
one gets the following result.

\begin{prop}
\label{letax}
 1) If $c\ge p(f+l)$, the pair (H,R) is an equilibrium,
 and moreover the strategy (R) is dominant for the police
(even strictly, if the previous inequality is strict). 2) If
$c<p(f+l)$ and $fp\le \overline{p}l$, the pair (H,C) is an
equilibrium and the strategy (H) is dominant (strictly if the
previous inequality is strict). 3) If $c<p(f+l)$, $fp>
\overline{p}l$, then the unique Nash equilibrium is the profile of
mixed strategies $(\beta,\overline{\beta}), (\alpha,
   \overline{\alpha})$.
\end{prop}

{\bf Remarks}. 1. Consequently, in cases 1) and 2) the actions of
the police are not effective. 2. It is not difficult to show that
the equilibrium in case 3) is stable.

It is more interesting to analyze the game obtained by extending the
strategy space of player I by allowing him to choose the amount
 $l$ of tax evasion:
$l\in[0,l_{M}]$, where $l_{M}$ is the full tax due to player I.
For example, we shall assume that the fine is proportional to $l$,
i.e. $f(l)=nl$. Say, in the Russian tax legislation $n=0.4$. Under
these assumptions the key coefficients $\alpha,\beta$ take the form

\[
\alpha=\frac{1}{p(n+1)}, \quad \beta=\frac{c}{l}\frac{1}{p(n+1)}.
\]
Let $H_{I}(l)$ denote the payoff to player I in the equilibrium when
$l$ is chosen. One can distinguish two cases:

 1) $p>\frac{1}{n+1}\Longleftrightarrow \alpha<1$.
If
\[
l>l_{1}=\frac{c}{p(n+1)} \Longleftrightarrow \beta <1,
\]
then $(\beta,\overline{\beta}),(\alpha,
   \overline{\alpha})$ is a stable equilibrium. If
$l<l_{1} \Longleftrightarrow \beta >1$, then $(H,R)$ is a stable
equilibrium. Since
\[
H_{I}(l<l_{1})=r+l<r+l_{1},
\]
\[
H_{I}(l>l_{1})=\beta\alpha(r+\overline{p}l-pf) +
\beta\overline{\alpha}(r+l) + \overline{\beta}\alpha r +
\overline{\beta}\overline{\alpha}r
  = r + \beta l(\alpha\overline{p}+
\overline{\alpha}-\alpha pn),
\]
it follows that $H_{I}(l>l_{1})>H_{I}(l<l_{1})$ would be possible
whenever $\beta l(\alpha\overline{p}+ \overline{\alpha}-\alpha
pn)\geq l_{1}$. However $\alpha\overline{p}+
\overline{\alpha}-\alpha pn=0$, hence this is not the case.
Consequently $H_{I}(l>l_{1})<H_{I}(l<l_{1})$ and therefore player I
will avoid tax on the amount $l=l_{1}$.

2) $p<\frac{1}{n+1}\Longleftrightarrow  \alpha>1$.

 If
 $l>l_{1}\Longleftrightarrow \beta <1 $,
then $(H,C)$ is an equilibrium, and if $l<l_{1} \Leftrightarrow
\beta >1$, then $(H,R)$ is an equilibrium. Since
\[
H_{I}(l<l_{1})=r+l<r+l_{1}, \quad H_{I}(l>l_{1})=r+(1-p)l-pln,
\]
the choice  $l>l_{1}$ is reasonable for player I as
\[
l(1-p-pn) \geq l_{1}.
\]
Hence $H_{I}(l>l_{1})>H_{I}(l<l_{1})$ whenever
$l\geq\frac{l_{1}}{1-p(n+1)}$.

Consequently, if
\begin{equation}
 \label{Inspect12}
\frac{l_{1}}{1-p(n+1)}\leq l _{M},
\end{equation}
the equilibrium strategy for player I is $l=l_{M}$ and otherwise
$l=l_{1}$.

 One can conclude that in both cases it is profitable
  to avoid tax on the amount $l_{1}$, but as the efficiency of tax man increases, it becomes unreasonable
  to avoid tax on a higher amount.

Let us see which condition in the second case would ensure the
inequality \eqref{Inspect12} when the amount of tax avoidance is
$l_{M}$ in the equilibrium. Plugging $l_1$ in \eqref{Inspect12}
yields
\[
\frac{c}{p(n+1)(1-p(n+1))}\leq l_{M}.
\]
Denoting $x=p(n+1)<1$ one can rewrite it as
\begin{equation}
\label{Inspect13} x^{2}-x+\frac{c}{l_{M}}\leq 0.
\end{equation}
The roots of the corresponding equation are
\[
x_{1,2}=\frac{1\pm\sqrt{1-\frac{4c}{l_{M}}}}{2}.
\]
Hence for $c>l_{M}/4$ inequality \eqref{Inspect12} does not hold for
any $p$, and for $c\leq l_{M}/4$ the solution to \eqref{Inspect13} is
\[
x\in\left[\frac{1-\sqrt{1-\frac{4c}{l_{M}}}}{2};\frac{1-\sqrt{1+\frac{4c}{l_{M}}}}{2}\right].
\]
Thus for
\begin{equation}
c\leq\frac{l_{M}}{4}, \quad p\in\left
[\frac{1-\sqrt{1-\frac{4c}{l_{M}}}}{2(n+1)};\frac{1-\sqrt{1+\frac{4c}{l_{M}}}}{2(n+1)}\right]
\end{equation}
it is profitable to avoid tax payment on the amount $l_{M}$.

Let us consider a numeric example with $n=0.4$ so that
$1/(n+1)=0.714$. If, say, $c=1000$, then $l_{M}=100000$.

1) If $p<0.714$, the condition $c\leq l_{M}/4$ holds true, as
$1000<2500$. Hence, for $p\in[0.007;0.707]$ it is profitable to
avoid tax on the whole amount, i.e. 100000.

2) If $p>0.714$ it is profitable to avoid tax on the amount
$l_{1}=714.29$.

Hence if the efficiency of tax payment checks is $p<0.707$, it is
profitable to avoid tax on the whole amount of 100000, and if
 $p>0.707$, then not more than on 1010.

\section{Cooperative games versus zero-sum games} \label{nmsolution}

 The aim of this Section is to present a curious
connection between competition and cooperation by linking the solutions to cooperative games and lower
values of auxiliary zero-sum games.
 The following definitions are standard:

{\it Cooperative game with non-transferable utility} is a triple
$G=(I,v,H)$, where $I = \{1,2,\ldots,n\}$ is the set of players, $H$
is a non-empty compact set from $R^n$, and $v$ is a mapping from the
set of all coalitions (non-empty subset $S \subset I$) to the set of
non-empty closed subsets $v(S)\subset H$.

 For $x,y \in H$ one says that $x$ {\it dominates} $y$ ($x\succ y$),
 if there exists a coalition $S\subseteq I$ such that

1)  $x,y\in v(S)$;

2)  $x_i>y_i$ for any $i\in S $.

A {\it (von Neumann - Morgenstern) solution} to $G = (I,v,H)$ is called a subset
$V\subseteq H$ such that

1) {\it (internal stability)} there are no pairs of vectors from $V$
such that one of them dominates the other one;

2) {\it (external stability)} for any $y\in H \backslash V$ there
exists $x\in V$ such that $x\succ y$.

For $A\subset R^n$ let us denote by $A_{\varepsilon}$ the
$\varepsilon$-neighborhood of $A$, i.e. $A_{\varepsilon} = A +
B_{\varepsilon}$ with
$$
B_{\varepsilon} = \left\{ x\in R^n\left|\sum\limits ^n_{i=1} x^2_i <
\varepsilon \right.\right\}.
$$

A closed subset $V$ is called a $\varepsilon$-{\it solution}
whenever $V$ is internally stable and for any $y\in H\backslash
V_{\varepsilon}$ there exists $x\in V$ such that $x\succ y$.

It is clear that if a closed $A$ is a $\varepsilon$-solution in the
game $G=(I,v,H)$ for any $\varepsilon
>0$, the $A$ is a solution.

Consider the function
$$
L(x,y) = \max\limits_{S:x,y\in v(S)}\ \min\limits_{i\in S}\,(x_i -
y_i).
$$
It is clear that $L(x,y)>0 \Leftrightarrow x\succ y$. Let ${\cal
A}(G) = \{B\in 2^H: B {\hbox{ замкнутое, }}\max\limits_{x,y\in
B}\,L(x,y)=0\}$.

\begin{prop}
 \label{nmpro1}
  Let $\varepsilon >0$. The game $G=(I,v,H)$ has a
$\varepsilon$-solution if and only if
 \begin{equation}
\label{nm1}
 \sup\limits_{A\in {\cal A}(G)}\{ \min\limits_{y\in
H\backslash A_{\scriptstyle\varepsilon}}\max\limits_{x\in A}
L(x,y)\}>0.
\end{equation}
\end{prop}

\begin{proof}
 Let \eqref{nm1} holds. Then there exist $A\subseteq H$ such that
 \begin{equation}
\label{nm2}
 \min\limits_{y\in{H\backslash
A_{\scriptstyle\varepsilon}}}\max\limits_{x\in A} L(x,y) > 0
 \end{equation}
and
 \begin{equation}
\label{nm3}
 \max\limits_{x,y\in A} L(x,y) = 0.
 \end{equation}
It follows from \eqref{nm2} that
 $\max\limits_{x\in A} L(x,{\overline y})> 0$
 for each ${\overline y}\in H\backslash A_{\varepsilon}$.
Consequently for arbitrary ${\overline y}\in H\backslash
A_{\varepsilon}$ one can find a ${\overline x}\in A$ such that
$L({\overline x},{\overline y}) > 0$, i.e. ${\overline x}\succ
{\overline y}$. It now follows from \eqref{nm3} that $L(x,y)\le 0$
for all $x,y\in A$, i.e. neither $x$ dominates $y$, nor vice versa.
Consequently the set $A$ is a $\varepsilon$-solution.

Now let $A$ is a $\varepsilon$-solution. Then $L(x,y)\le 0$ for all
$x,y\in A$, but $L(x,x) = 0$ implying \eqref{nm3}.

With $A$ being a $\varepsilon$-solution, for any ${\overline y}\in
H\backslash A_{\varepsilon}$ there exists a ${\overline x}\in A$
such that ${\overline x}\succ {\overline y}$, i.e. $L({\overline
x},{\overline y})>0$. Consequently $\max\limits_{x\in A}
L(x,{\overline y})>0$ for any ${\overline y}\in H\backslash
A_{\varepsilon}$ and therefore
\[
\min\limits_{y\in H\backslash
A_{\scriptstyle\varepsilon}}\max\limits_{x\in A} L(x,y)>0 .
\]
\end{proof}

Let us choose now a $\varepsilon >0$ and introduce a two-person
zero-sum game $\Gamma_{\varepsilon }(I,v,H) $, in which the first
player makes a first move by choosing a $A\in \cal A$ and then the
second player replies by choosing $y\in H\backslash
A_{\varepsilon}$. Finally the first player makes the third (and the
last) move by choosing $x\in A$. The income of the first player in
this game equals $L(x,y)$ (the second one gets $- L(x,y)$).

The following statement is a direct corollary of Proposition \ref{nmpro1}.

\begin{prop}
 \label{nmpro2}
 The game $G = (I,v,H) $ has a $\varepsilon$-solution if and only if
 the maximal guaranteed gain of the first plater in
  $\Gamma_{\varepsilon }(I,v,H) $ is positive.
  Moreover any
$A\in \cal A $ yielding such gain is a $\varepsilon$-solution to
$G$.
\end{prop}

One can now define a two-player zero-sum game in the normal form
$N_{\varepsilon }(I,v,H)$, where the strategies of the first players
are the sets $A\in \cal A$ and the strategies of the second player
are the mappings $f: A\in {\cal A} \mapsto H\backslash
A_{\varepsilon}$. Let $F$ denote the set of the strategies of the
second player. Let the payoff to the first player in this game
$N_{\varepsilon }(I,v,H)$ be
$$
h(A,f) = \max\limits_{x\in A} L(x,f(A)).
$$

\begin{prop}
 \label{nmpro3}
 The game $N_{\varepsilon }(I,v,H)$ has a value, i.e.
$$
\sup\limits_{A\in\cal A}\inf\limits_{f\in F} h(A,f) =
\inf\limits_{f\in F}\sup\limits_{A\in \cal A} h(A,f).
$$
\end{prop}
\begin{proof}
 Consider the strategy $f^*$ such that
$$
\min\limits_{y\in H\backslash
A_{\scriptstyle\varepsilon}}\max\limits_{x\in A} L(x,y) =
\max\limits_{x\in A} L(x,f^*(A))
$$
for any $A\in \cal A$.

Then
$$
\sup\limits_{A\in \cal A}\inf\limits_{f\in F} h(A,f)\le
\inf\limits_{f\in F}\sup\limits_{A\in \cal A} h(A,f)\le
\sup\limits_{A\in \cal A} h(A,f^*) =
$$
$$
= \sup\limits_{A\in \cal A}\max\limits_{x\in A} L(x,f^*(A)) =
\sup\limits_{A\in \cal A}\min\limits_{y\in H\backslash
A_{\scriptstyle\varepsilon}}\max\limits_{x\in A} L(x,y)=
 \sup\limits_{A\in \cal A}\inf\limits_{f\in F} h(A,f).
$$
\end{proof}

\begin{theorem} (\cite{Ma1})
 \label{nmpro4}
 The game $G = (I,v,H)$ has a $\varepsilon$-solution if and only if
$$
 \sup\limits_{A\in\cal A}\inf\limits_{f\in F} h(A,f)>0
$$
holds for the game $N_{\varepsilon }(I,v,H)$.
\end{theorem}

\begin{proof}
Follows from Proposition \ref{nmpro1}, \ref{nmpro3}.
\end{proof}

\section{Stability for two-action multi-agent games}

One of the most popular game in recent game-theoretic literature is the so called minority game, see \cite{Coo}.
It represents a particular game of many agents with each agent having two strategies.
With this motivation, we are going to analyze here the stability property of games with two actions of each player.
 
To prepare the stage for our analysis, let us consider a general non-symmetric game $\Gamma$ of $m$ agents, where each agent $j$ has $n_j$ strategies $\{s_j^k\}$,
 $k=1,\cdots, n_j$, and receives payoffs $\Pi_j(s_1^{j_1}, s_2^{j_2}, \cdots, s_m^{j_m})$ at a profile
 $(s_1^{j_1}, s_2^{j_2}, \cdots, s_m^{j_m})$. Then, for a mixed strategy profile $\sigma=(\sigma_1, \sigma_2,...,\sigma_m)$ with
  \[
\sigma_1=(x_1^1,...,x_{n_1}^1), \quad
\sigma_2=(x_1^2,...,x_{n_2}^2), \quad ..., \quad
\sigma_m=(x_1^m,...,x_{n_m}^m)
\]
the payoff of agent $j$ becomes
\begin{equation}
 \label{220}
\Pi_i(\sigma_1, \sigma_2,...,\sigma_m)
 =\sum_{j_1=1}^{n_1}\sum_{j_2=1}^{n_2}...\sum_{j_m=1}^{n_m}
x^1_{j_1}x^2_{j_2}...x^m_{j_m}\Pi_i(s_1^{j_1},s_2^{j_2},...,s_m^{j_m}).
\end{equation}

 The Replicator Dynamics (RD) equations have the form

\begin{equation}
\label{247} \dot x_i^j= (\Pi_j(s^i_j,\sigma_{-j})-\Pi_j
(\sigma_j,\sigma_{-j})) x_i^j, \quad i=1,...,n_j,\, j=1,...,m,
\end{equation}
where $\sigma_{-j}$ denotes the collection of the strategies of all players in $\sigma$
others than $j$. This system describes an evolution of the behavior
of the players applying a try-and-error method of shifting the
strategies in the direction of a better payoff. As is well known, a Nash equilibrium is a fixed point of system \eqref{247}. The
r.h.s of \eqref{247} is sometimes called the Nash vector field of the
game.

Recall that a Nash equilibrium for a game $\Gamma$ is called {\it asymptotically
 stable, neutrally stable or unstable in the Lyapunov sense\index{Lyapunov stability} (or dynamically)} if it
 so for the corresponding dynamics \eqref{247}. (This means
 roughly speaking that if starting with strategies near the
 equilibrium, players would adjust their strategies in the direction of
 better payoffs, their strategies would converge to this equilibrium.)
On the other hand, a Nash equilibrium $(\sigma_1,...,\sigma_m)$ in such game is called {\it structurally
stable}\index{structural stability} if for arbitrary $\epsilon >0$
there exists a $\delta>0$ such that for all games $\tilde \Gamma$
with the same number of players and pure strategies and with payoffs
$\tilde \Pi$ that differ from $\Pi$ no more than by $\delta$, i.e.
such that
\begin{equation}
\label{263}
 |\Pi_i(s_1^{j_1},s_2^{j_2},...,s_m^{j_m})
 -\tilde \Pi_i(s_1^{j_1},s_2^{j_2},...,s_m^{j_m})|<\delta
 \end{equation}
 for all $i,s_1^{j_1},s_2^{j_2},...,s_m^{j_m}$, there
exists a Nash equilibrium $(\tilde \sigma_1,...,\tilde \sigma_m)$
for the game $\tilde \Gamma$ such that $|\tilde
\sigma_j-\sigma_j|<\epsilon$ for all $j=1,...,m$.
It makes sense also to speak about structural stability of
dynamically stable or unstable equilibria, i.e. a dynamically stable
(or unstable) Nash equilibrium $(\sigma_1,...,\sigma_m)$ in a game
$\Gamma$ is called {\it structurally stable} if for arbitrary
$\epsilon
>0$ there exists a $\delta>0$ such that for all games $\tilde
\Gamma$ with the same number of players and pure strategies and with
payoffs $\tilde \Pi$ that differ from $\Pi$ no more than by $\delta$
(i.e. \eqref{263} holds), there exists a dynamically stable
(respectively unstable) Nash equilibrium $(\tilde
\sigma_1,...,\tilde \sigma_m)$ for the game $\tilde \Gamma$ such
that $|\tilde \sigma_j-\sigma_j|<\epsilon$ for all $j=1,...,m$.

 The notion of stability is closely related to another important notion
of a {\it generic property}: a property
(object or characteristics) in a class of structures parametrized by
a collection of real numbers $s$ from a given subset $S$ of a
Euclidean space is called {\it generic} if it holds for $s$ from a
subset $\tilde S\subset S$ that is both open (which means that if
$s_0\in \tilde S$, then all $s\in S$ that are closed enough to $s_0$
belong to $\tilde S$ as well, i.e. the property of being in $\tilde
S$ is structurally stable in any point $s\in \tilde S$) and dense
(which means that for any $s_0$ there exists an $s\in \tilde S$ that
is arbitrary close to $s_0$, i.e. the negation of being in $\tilde
S$ is nowhere structurally stable).

As mentioned above, we are going to concentrate on the class
$\Gamma_n^2$ of mixed strategy extensions of games of $n$ players
each having only two strategies.

Let $A^i_{j_1,...,j_n}$ denote the
payoff to $i$ under pure profile $\{j_1,...,j_n\}$, $j_k=1,2$. A
mixed strategy profile can be described by families
\[
\sigma_1=(x_1,1-x_1), \sigma_2=(x_2,1-x_2),...,
\sigma_n=(x_n,1-x_n).
\]

Equations \eqref{247} can be written in terms of $x_1,...,x_n$
yielding (check it!)
\begin{equation}
\label{264} \dot x_i=x_i(1-x_i) \sum _{I\in \{1,...,n\}\setminus i}
 \tilde A^i_I \prod_{k\in I}x_k\prod_{k\notin I}(1-x_k), \quad
 i=1,...,n,
\end{equation}
where
\[
 \tilde A^i_I=
 A^i_{j_1...j_{i-1}1j_{i+1}...j_n}-A^i_{j_1...j_{i-1}2j_{i+1}...j_n}
 \]
 with $j_k=1$ whenever $k\in I$ and $j_k=2$ otherwise.
 Hence pure mixed (i.e. with all probabilities being positive) Nash
equilibria for a game in $\Gamma_n^2$ are given by vectors
$x^{\star}=(x_1^{\star},...,x_n^{\star})$ with coordinates from
$(0,1)$ solving the following system of $n$ equations
\begin{equation}
\label{265} \sum _{I\in \{1,...,n\}\setminus i}
 \tilde A^i_I \prod_{k\in I}x_k\prod_{k\notin I}(1-x_k)=0, \quad
 i=1,...,n.
\end{equation}
 In particular
 for $n=3$, denoting $x_1,x_2,x_3$ by $x,y,z$ and arrows of payoffs $A^1$,
 $A^2$, $A^3$ by $A,B,C$ yields for system \eqref{264} the following explicit form
\begin{equation}
\label{266}
  \left\{
  \begin{array}{lr}
  \dot x=x(1-x)(a+A_2y+A_3z+Ayz), \quad 0\le
x\le 1 \\
 \dot y=y(1-y)(b+B_1x+B_3z+Bxz), \quad 0\le y\le 1 \\
\dot z=z(1-z)(c+C_1x+C_2y+Cxy), \quad 0\le z\le 1
 \end{array}
 \right.
\end{equation}
as well as the form
\begin{equation}
\label{2661}
\left\{
 \begin{array}{lr}
  a+A_2y+A_3z+Ayz=0 \\
  b+B_1x+B_3z+Bxz =0 \\
 c+C_1x+C_2y+Cxy = 0
 \end{array}
 \right.
\end{equation}
for system \eqref{265},
 where
\[
a=A_{122}-A_{222}, \quad A_2=A_{112}-A_{212}-a, \,
A_3=A_{121}-A_{221}-a,
  \]
\[
A=A_{111}-A_{211}-a-A_2-A_3,
 \]
 and the coefficients in other two lines are defined analogously.

 Assuming $x^{\star}=(x_1^{\star},...,x_n^{\star})$
 solves \eqref{265}, it is  convenient to rewrite system \eqref{264} in terms of
 the deviations from the equilibrium
 $\xi_i=x_i-x_i^{\star}$. One then sees by inspection that the matrix of linear
 approximation (the Jacobian matrix) $J^{\star}$ has the entries
\begin{equation}
\label{267} J^{\star}_{ij}=x_i^{\star}(1-x_i^{\star})\sum _{I\in
\{1,...,n\}\setminus \{i,j\}}
 (\tilde A^i_{I\cup j}-\tilde A^i_I) \prod_{k\in I}x_k^{\star}\prod_{k\notin (I\cup \{i,j\})}(1-x_k^{\star})
 \end{equation}
 for $i\neq j$ and $J^{\star}_{ii}=0$ for all $i$.

\begin{theorem}
\label{t109} Let $J^{\star}$ be the Jacobian matrix (described in
Exercise 9.12) of a pure mixed equilibrium
$x^{\star}=(x_1^{\star},...,x_n^{\star})$
 solving \eqref{265}. If at least one of the eigenvalues of $J^{\star}$
 has a non-vanishing real part, then $x^{\star}$ is unstable in
 Lyapunov sense (i.e. dynamically). In particular, if $n$ is an odd
 number, then a necessary condition for the Liapunov stability of
 $x^{\star}$ is the degeneracy of $J^{\star}$, that is $\det
 J^{\star}=0$.
\end{theorem}

\begin{proof} From \eqref{267} one deduces that $J^{\star}$ has zeros on the
main diagonal. Hence the sum of its eigenvalues vanishes.
Consequently if there exists an eigenvalue with a non-vanishing real
part there should necessarily exist also an eigenvalue with a
positive real part, which implies instability. As eigenvalues with
vanishing real parts appear as pairs of conjugate imaginary numbers,
it follows that in case of odd $n$ the fact that all real parts
vanish implies that zero should be an eigenvalue, i.e. the
degeneracy of $J^{\star}$.
\end{proof}

Let us write down the condition $\det
 J^{\star}=0$ explicitly (i.e. in terms of the payoff coefficients) for the case
 $n=3$.

 (i) From \eqref{267} the condition $\det
 J^{\star}=0$ writes down as
\begin{equation}
\label{268}
 (A_2+Az^{\star})(B_3+Bx^{\star})(C_1+Cy^{\star})
  +(B_1+Bz^{\star})(C_2+Cx^{\star})(A_3+Ay^{\star})=0.
 \end{equation}

 (ii) Solving \eqref{2661}
 by expressing $y,z,$ in terms of $x$ and putting this in the first equation
 leads to the quadratic equation:
 \begin{equation}
\label{269}
 v(x^{\star})^2+ux^{\star}+w=0,
 \end{equation}
 where
 \[
 w=aC_2B_3+cbA-bA_3C_2-cA_2B_3, v=aBC+AB_1C_1-BA_2C_1-CA_3B_1,
 \]
 \[
 u=a(BC_2+CB_3)+b(AC_1-CA_3)+c(AB_1-BA_2)-A_2B_3C_1-A_3B_1C_2.
 \]

 (iii) Using the system that is solved by
 $x^{\star},y^{\star},z^{\star}$ to express $y^{\star}z^{\star}$
 as a linear function of $y^{\star},z^{\star}$ allows to rewrite
 \eqref{268} as a quadratic equation
 (in $x^{\star},y^{\star},z^{\star}$). Expressing the quadratic terms of
 this equation again via linear terms leads to the equation on $x^{\star}$ only:
 \[
 u+2x^{\star}v=0.
 \]
 Comparing this with \eqref{269} leads to the conclusion that
 \eqref{268} is equivalent to the equation
\begin{equation}
\label{270}
 u^2-4vw=0,
 \end{equation}
 which is a polynomial homogeneous equation of the sixth order in
 coefficients $a,A,A_2,A_3$, $b,B,B_1,B_3$, $c,C,C_1,C_2$.

 {\bf Corollary.} The property to have a dynamically unstable pure
 mixed equilibrium is generic among the games of type $\Gamma^2_3$
 that have pure mixed equilibria. More precisely, apart from the
 games from the (algebraic of the sixth order) manifold $M$ described
 by equation \eqref{270} pure mixed equilibria are always dynamically
 unstable and structurally stable (as dynamically unstable
 equilibria).

\begin{prop}
 Under the assumption
\eqref{270} (or equivalently \eqref{268}) there exist $\alpha, \beta,
\gamma$
  such that a function $V$ of the (relative
 entropy) form
 \begin{equation}
\label{271}
 \alpha [x^{\star} \ln x +(1-x^{\star}) \ln (1-x)]
 +\beta [y^{\star} \ln y +(1-y^{\star}) \ln (1-y)]
 +\gamma [z^{\star} \ln z +(1-z^{\star}) \ln (1-z)]
 \end{equation}
is a first integral for \eqref{266}, i.e. it does not change along
the trajectories of \eqref{266} ($dV/dt=0$) if and only if the
condition
\begin{equation}
\label{272}
 AB_1C_1+aBC= BA_2C_1+CA_3B_1
 \end{equation}
 holds. In the latter case
 \[
 V=(B_1+Bz^{\star})(C_1+Cy^{\star}) [x^{\star} \ln x +(1-x^{\star}) \ln (1-x)]
 \]
 \[
 -(A_2+Az^{\star})(C_1+Cy^{\star}) [y^{\star} \ln y +(1-y^{\star}) \ln (1-y)]
 \]
 \[
 -(A_3+Ay^{\star})(B_1+Bz^{\star}) [z^{\star} \ln z +(1-z^{\star}) \ln (1-z)]
 \]
 is an integral.
 \end{prop}

 \begin{proof}
  Substituting \eqref{271} in the equation $dV/dt=0$
 that \eqref{271} is an integral of motion
 \eqref{266} if and only if
 \[
 \left\{
  \begin{array}{lr}
  \al (A_2+Az)+\beta(B_1+Bz)=0 \\
   \al (A_3+Ay)+ \gamma (C_1+Cy) =0 \\
   \beta (B_3+Bx)+\gamma (C_2+Cy)=0 \\
 \al A +\beta B +\gamma C=0.
 \end{array}
 \right.
  \]
 Expressing $\beta, \gamma$ in terms of $\alpha$ from the first two
 equations one observes that the third equation is then automatically
 satisfied due to \eqref{268}. Solving the last equation leads to
 \eqref{272}.
 \end{proof}

Having a first integral $V$ as above, allows one to conclude that
 the equilibrium $(x^{\star}, y^{\star}, z^{\star})$ is neutrally stable
 in the Lyapunov sense.
 
 \section{Nonlinear Markov games}

Nonlinear Markov
games arise as a (competitive) controlled version of
nonlinear Markov processes (an emerging field of intensive research,
see e.g. \cite{Ko10}, \cite{Frank} and references therein). This class of
games can model a variety of situation for economics and epidemics,
statistical physics, and pursuit - evasion processes.

A discrete-time, discrete-space {\it nonlinear Markov
semigroup} $\Phi^k$, $k\in\N$, is
specified by an arbitrary continuous mapping $\Phi: \Si_n \to
\Si_n$, where the simplex
\[
\Si_n=\{\mu=(\mu_1,...,\mu_n)\in \R^n_+: \, \sum_{i=1}^n\mu_i=1\}
\]
represents the set of probability laws on the finite state space
$\{1,...,n\}$. For a measure $\mu \in \Si_n$ the family
$\mu^k=\Phi^k\mu$ can be considered an evolution of measures on
$\{1,...,n\}$. But it does not yet define a random process, because
finite-dimensional distributions are not specified. In order to
obtain a process we have to choose a {\it stochastic representation}
for $\Phi$, i.e. to write it down in the form

\begin{equation}
\label{eqstochrepfordiscreteMchains}
\Phi(\mu)=\{\Phi_j(\mu)\}_{j=1}^n=\{\sum_{i=1}^nP_{ij}(\mu)\mu_i\}_{j=1}^n,
\end{equation}
where $P_{ij}(\mu)$ is a family of stochastic
matrices
depending on $\mu$ (nonlinearity!), whose elements specify the {\it
nonlinear transition probabilities}.
 For any given $\Phi: \Si_n
\mapsto \Si_n$ a representation \eqref{eqstochrepfordiscreteMchains}
exists but is not unique. There exists a unique
representation \eqref{eqstochrepfordiscreteMchains} with the
additional condition that all matrices $P_{ij}(\mu)$ are one dimensional:
\begin{equation}
\label{eqstochrepfordiscreteMchainssimple}
P_{ij}(\mu)=\Phi_j(\mu), \quad i,j=1,...,n.
\end{equation}
Once a stochastic representation
\eqref{eqstochrepfordiscreteMchains} for a mapping $\Phi$ is chosen
we can naturally define, for any initial probability law
$\mu=\mu^0$, a stochastic process $i_l$, $l\in \Z_+$, called a {\it
nonlinear Markov chain}\index{nonlinear Markov chain}, on
$\{1,...,n\}$ in the following way. Starting with an initial
position $i_0$ distributed according to $\mu$ we then choose the
next point $i_1$ according to the law $\{P_{i_0j}(\mu)\}_{j=1}^n$,
the distribution of $i_1$ becoming $\mu^1=\Phi(\mu)$:
\[
\mu^1_j=\P(i_1=j)=\sum_{i=1}^n P_{ij} (\mu) \mu_i =\Phi_j(\mu).
\]
Then we choose $i_2$ according to the law
$\{P_{i_1j}(\mu^1)\}_{j=1}^n$, and so on. The law of this process at
any given time $k$ is $\mu^k=\Phi^k(\mu)$; that is, it is given by
the semigroup. However, now the finite-dimensional distributions are
defined as well. Namely, say for a function $f$ of two discrete
variables, we have
\[
\E f(i_k,i_{k+1})=\sum_{i,j=1}^n f(i,j) \mu^k_iP_{ij}(\mu^k).
\]
In other words, this process can be defined as a time nonhomogeneous
Markov chain with transition probabilities $P_{ij}(\mu^k)$ at time $t=k$.

We turn now to nonlinear chains in continuous time. A {\it nonlinear
Markov semigroup}\index{nonlinear Markov semigroup} in continuous
time and with finite state space $\{1,...,n\}$ is defined as a
semigroup $\Phi^t$, $t\ge 0$, of continuous transformations of
$\Si_n$. As in the case of discrete time the semigroup itself does
not specify a process. To get a process, assume the semigroup $\Phi^t$ is differentiable in $t$, so that we
can define the {\it (nonlinear) infinitesimal
generator} of the
semigroup $\Phi^t$ as the nonlinear operator on measures given by
\[
 A(\mu)=\frac{d}{dt}\Phi^t|_{t=0}(\mu).
 \]
 The semigroup identity for $\Phi^t$ implies that
 $\Phi^t(\mu)$ solves the Cauchy problem
\begin{equation}
\label{eqevoleqfornonlMarkovch}
 \frac{d}{dt} \Phi^t
(\mu)=A(\Phi^t(\mu)), \quad \Phi^0(\mu)=\mu.
\end{equation}

As follows from the invariance of $\Si_n$ under these dynamics, the
mapping $A$ is {\it conditionally positive}\index{conditional
positivity} in the sense that $\mu_i =0$ for a $\mu \in \Si_n$
implies $A_i(\mu) \ge 0$ and is also {\it conservative}\index{conservativity} in the sense
that $A$ maps the measures from $\Si_n$ to the space of signed
measures
\[
\Si^0_n=\{\nu \in \R^n: \sum_{i=1}^n \nu_i=0\}.
\]
We shall say that such a generator $A$ has a {\it stochastic representation}
 if it
can be written in the form
\begin{equation}
\label{eqstochrepforgenMch}
 A_j(\mu)=\sum_{i=1}^n\mu_iQ_{ij}(\mu)=(\mu Q(\mu))_j,
\end{equation}
where $Q(\mu)=\{Q_{ij}(\mu)\}$ is a family of infinitesimally
stochastic matrices (or $Q$=matrices) \index{stochastic matrix!infinitesimally}
depending on $\mu \in \Si_n$. Thus
in its stochastic representation the generator has the form of a
usual Markov chain generator, though depending additionally on the
present distribution. The existence of a stochastic representation
for the generator is not
difficult to obtain, see \cite{Ko10}.

The examples of nonlinear Markov chains are numerous including Lotka-Volterra systems, general replicator
dynamics of the evolutionary game theory, models of epidemics, coagulation processes, see more in \cite{Ko10}.

 Now we discuss the corresponding nonlinear
extension of controlled processes.

Nonlinear Markov games can be considered as a systematic tool for
modeling  deception. In particular, in a game of pursuit - evasion,
an evading object can create false objectives or hide in order to
deceive the pursuit. Thus, observing this object leads not to its
precise location, but to its distribution only, implying that it is
necessary to build competitive control on the basis of the
distribution of the present state. Moreover, by observing the action
of the evading objects, one can make conclusions about its certain
dynamic characteristics making the (predicted) transition
probabilities depending on the observed distribution via these
characteristics. This is precisely the type of situations modeled by
nonlinear Markov games.

The starting point for the analysis is the observation that a
nonlinear Markov semigroup is after all just a deterministic dynamic
system (though on a weird state space of measures and with a specifically structured payoff function).
Thus, as the stochastic control theory is a natural extension of the
deterministic control, we are going to further extend it by turning
back to deterministic control, but of measures, thus exemplifying
the usual spiral development of science. The next 'turn of the
screw' would lead to stochastic measure-valued games forming a
stochastic control counterpart for the class of processes discussed
in the previous section.

 We shall work directly in the
competitive control setting (game theory), which of course includes
the usual optimization as a particular case, but for simplicity only
in discrete time and finite original state space $\{1,...,n\}$. The
full state space is then chosen as a set of probability measures
$\Si_n$ on $\{1,...,n\}$.

 Suppose we are given two metric spaces $U$, $V$ of the control parameters of two
players, a continuous transition cost function $g(u,v,\mu)$, $u\in
U$, $v\in V$, $\mu \in \Si_n$ and a transition law $\nu (u,v,\mu)$
prescribing the new state $\nu\in \Si_n$ obtained from $\mu$ once
the players had chosen their strategies $u\in U, v\in V$. The
problem of the corresponding one-step game (with sequential moves)
consists in calculating the Bellman operator
\begin{equation}
 \label{eqnonlineargameBeloper}
(BS)(\mu)= \min_u\max_v [g(u,v,\mu)+S(\nu(u,v,\mu))]
\end{equation}
for a given final cost function $S$ on $\Si_n$. According to the
dynamic programming principle (see e.g. \cite{KoMabook}), the dynamic multi-step game
solution is given by the iterations $B^kS$. Often of interest is the
behavior of this optimal cost $B^kS(\mu)$ as the number of steps $k$
go to infinity.

The function $\nu (u,v,\mu)$ can be interpreted as the controlled
version of the mapping $\Phi$ specifying a nonlinear discrete time
Markov semigroup.
Assume a stochastic representation for this mapping is chosen, i.e.
\[
\nu_j (u,v,\mu)=\sum_{i=1}^n\mu_iP_{ij}(u,v,\mu)
\]
with a given family of (controlled) stochastic matrices $P_{ij}$.
Then it is natural to assume $g$ to describe the average over the
random transitions, i.e. be given by
\[
g(u,v,\mu)=\sum_{i,j=1}^n\mu_iP_{ij}(u,v,\mu)g_{ij}
\]
with certain real coefficients $g_{ij}$. Under this assumption the
Bellman operator \eqref{eqnonlineargameBeloper} takes the form
\begin{equation}
 \label{eqnonlineargameBeloper1}
(BS)(\mu)= \min_u\max_v
[\sum_{i,j=1}^n\mu_iP_{ij}(u,v,\mu)g_{ij}+S\left(\sum_{i=1}^n\mu_iP_{i.}(u,v,\mu)\right)].
\end{equation}

We can now identify the (not so obvious) place of the usual
stochastic control theory in this nonlinear setting. Namely, assume
$P_{ij}$ above do not depend on $\mu$. But even then the set of the
linear functions $S(\mu)=\sum_{i=1}^ns_i\mu^i$ on measures
(identified with the set of vectors $S=(s_1,...,s_n)$) is not
invariant under $B$. Hence we are not automatically reduced to the
usual stochastic control setting, but to a game with incomplete
information, where the states are probability laws on $\{1,...,n\}$,
i.e. when choosing a move the players do not know the position
precisely, but only its distribution. Only if we allow only Dirac
measures $\mu$ as a state space (i.e. no uncertainty on the state),
the Bellman operator would be reduced to the usual one of the
stochastic game theory:
\begin{equation}
 \label{eqnonlineargameBeloperredlin}
(\bar B S)_i= \min_u\max_v \sum_{j=1}^nP_{ij}(u,v)(g_{ij}+S_j).
\end{equation}

As an example of a nonlinear result we shall get here an analog of
the result on the existence of the average income  for long lasting
games (see \cite{Ko10} for a proof).

\begin{prop}
\label{propnonlgameeigenvalue}
 If the mapping $\nu$ is a contraction uniformly in $u,v$, i.e. if
\begin{equation}
 \label{eqtransitioncontracts}
\|\nu (u,v,\mu^1)-\nu (u,v,\mu^2)\| \le \de \|\mu^1-\mu^2\|
\end{equation}
with a $\de \in (0,1)$, where $\|\nu\|=\sum_{i=1}^n|\nu_i|$, and if
$g$ is Lipschitz continuous, i.e.
\begin{equation}
 \label{eqnonlcostLip}
\|g (u,v,\mu^1)-g (u,v,\mu^2)\| \le C \|\mu^1-\mu^2\|
\end{equation}
with a constant $C>0$,
 then there exists a unique
$\la\in\R$ and a Lipschitz continuous function $S$ on $\Si_n$ such
that
\begin{equation}
\label{eqnonleigenvector}
 B(S)=\la+S,
\end{equation}
and for all $g\in C(\Si_n)$ we have
\begin{align}
\label{eqnonlgamelongtimebeh}
&\|B^mg-m\la\|\le\|S\|+\|S-g\|,\\
&\lim_{m\to\infty}\frac{B^mg}{m}=\la.
\end{align}
\end{prop}

One can extend the other results for stochastic multi-step games to
this nonlinear setting, say, the turnpike theorems from
 \cite{KoMabook}, and then go on studying the
nonlinear Markov analogs of differential games.

\section{Colored options as a game against Nature}
\label{Options}

Here we introduce a game-theoretic analysis of rainbow options in incomplete markets.
Further developments can be found in \cite{MaMu}, \cite{BerAuKo} and \cite{Ko11}.

Recall that a European option is a contract between
two parties where one party has right to complete a transaction in
the future (with previously agreed amount, date and price) if he/
she chooses, but is not obliged to do so.
More precisely, consider a financial market dealing with several securities:
the risk-free bonds (or bank account) and $J$ common stocks,
$J=1,2...$. In case $J>1$, the corresponding options are called {\it
colored or rainbow options}
($J$-colors option for a given $J$). Suppose the prices of the units of these
securities, $B_m$ and $S_m^{i}$, $i\in \{1,2,...,J\}$, change in
discrete moments of time $m=1,2,...$
according to the recurrent equations $B_{m+1}=\rho B_m$, where the
$\rho \geq 1$ is an interest rate which remains unchanged over time,
and $S_{m+1}^{i}=\xi_{m+1}^{i}S_m^i$, where $\xi_m^{i},i\in
\{1,2,...,J\}$, are unknown sequences taking values in some fixed
intervals $M_{i}=[d_{i},u_{i}]\subset \R$. This model
generalizes the colored version of the classical CRR model in a
natural way. In the latter a sequence $\xi_m^{i}$ is confined to
take values only among two boundary points $d_{i},u_{i}$, and it is
supposed to be random with some given distribution. In our model any
value in the interval $[d_{i},u_{i}]$ is allowed and no probabilistic assumptions are made.

The type of an option is specified by a given premium function $f$
of $J$ variables. The following are the standard examples:

option delivering the best of $J$ risky assets and cash
\begin{equation}
f(S^{1},S^{2},...,S^{J})=\max (S^{1},S^{2},...,S^{J},K),
 \label{best risky assets}
\end{equation}

calls on the maximum of $J$ risky assets
\begin{equation}
f(S^{1},S^{2},...,S^{J})=\max (0,\max (S^{1},S^{2},...,S^{J})-K),
\label{calls on max}
\end{equation}

multiple-strike\index{option!multiple-strike} options
\begin{equation}
f(S^{1},S^{2},...,S^{J})=\max
(0,S^{1}-K_{1},S^{2}-K_{2},....,S^{J}-K_{J}),
\label{multiple strike options}
\end{equation}

portfolio\index{option!portfolio} options
\begin{equation}
f(S^{1},S^{2},...,S^{J})=\max
(0,n_{1}S^{1}+n_{2}S^{2}+...+n_{J}S^{J}-K),
\end{equation}

and spread\index{option!spread} options
\begin{equation}
f(S^{1},S^{2})=\max (0,(S^{2}-S^{1})-K).
\end{equation}

Here, the $S^{1},S^{2},...,S^{J}$ represent the (in principle
unknown at the start) expiration date values of the underlying
assets, and $K,K_{1},...,K_{J}$ represent the (agreed from the
beginning) strike prices. The presence of $\max$ in all these
formulae reflects the basic assumption that the buyer is not obliged
to exercise his/her right and would do it only in case of a positive
gain.

The investor is supposed to control the growth of his/her capital in the
following way. Let $X_m$ denote the capital of the investor at the
time $m=1,2,...$. At each time $m-1$ the investor determines his
portfolio by choosing the numbers $\ga _m^i$ of common stocks
of each kind to be held so that the structure of the capital is
represented by the formula

\[
X_{m-1} =\sum_{j=1}^{J}\ga_m^j S_{m-1}^j
 +(X_{m-1}-\sum_{j=1}^{J}\ga_m^j S_{m-1}^j),
\]
where the expression in brackets corresponds to the part of his
capital held in the bank account. The control parameters $\ga_m^j$
 can take all real values, i.e. short selling and borrowing are
allowed. The value $\xi_m$ becomes known at the moment $m$ and thus
the capital at the moment $m$ becomes

\begin{equation}
\label{eqnewcap1}
X_m=\sum_{j=1}^{J}\ga _m^j\xi_m^j S_{m-1}^j +\rho
(X_{m-1}-\sum_{j=1}^J \ga_m^j S_{m-1}^j),
\end{equation}
if transaction costs are not taken into account.

If $n$ is the prescribed {\it maturity date}\index{option!maturity
date}, then this procedures repeats $n$ times starting from some
initial capital $X=X_0$ (selling price of an option) and at the end
the investor is obliged to pay the premium $f$ to the buyer. Thus
the (final) income of the investor equals
\begin{equation}
\label{eqnewcap2}
G(X_n,S_n^1,S_n^2,...,S_n^J)=X_n-f(S_n^1,S_n^2,...,S_n^J).
\end{equation}

The evolution of the capital can thus be described by the $n$-step
game of the investor with Nature, the behavior of the latter
being characterized by unknown parameters $\xi_m^j$. The strategy of
the investor is by definition any sequences of vectors
$(\ga_1,\cdots,\ga_n)$
 such that each
$\ga_m$ could be chosen using the whole previous
information: the sequences $X_{0},...,X_{m-1}$ and $S_0^i,...,S_{m-1}^j$
(for every stock $j=1,2,...,J$).
 The control parameters $\ga_m^j$ can take all real values,
i.e. short selling and borrowing are allowed. A position of the game
at any time $m$ is characterized by $J+1$ non-negative numbers
$X_m,S_m^1, \cdots ,S_m^J$ with the final income specified by the
function
\begin{equation}
G(X,S^{1},...,S^{J})=X-f(S^{1},...,S^{J}).
\label{eqG function}
\end{equation}

The main definition of the theory is as follows. A strategy $\ga_1,\cdots,\ga_n$, of the investor is called a
{\it hedge}\index{hedge}, if for any sequence
  $(\xi_1, \cdots ,\xi_n)$
the investor is able to meet his/her obligations, i.e.
\[
G(X_{n},S_{n}^{1},...,S_{n}^{J})\geq 0.
\]
The minimal value of the capital $X_{0}$ for which the hedge exists
is called the {\it hedging price} $H$ of an option.

Looking for the guaranteed payoffs means looking
for the worst-case scenario (so called {\it robust-control
approach}\index{robust control}), i.e. for the minimax strategies. Thus if the
final income is specified by a function $G$, the guaranteed income
of the investor in a one-step game with the initial conditions
$X,S^{1},...,S^{J}$ is given by the {\it Bellman
operator}\index{Bellman operator}
\begin{equation}
\label{eqBellmanforopnonred}
\mathbf{B}G(X, S^1,\cdots, S^J)
 =\frac{1}{\rho}\max_{\ga}\min_{\{\xi^j \in [d_j,u_j]\}}
 G(\rho X+ \sum_{i=1}^{J}\ga^i\xi^iS^{i}
  -\rho \sum_{i=1}^{J}\ga^iS^i,\xi^1S^1, \cdots,\xi^JS^J),
\end{equation}
and (as it follows from the standard backward induction argument) the guaranteed income of
the investor in the $n$-step game with the initial conditions
$X_{0},S_{0}^{1},...,S_{0}^{J}$ is given by the formula

\[
\mathbf{B}^{n}G(X_{0},S_{0}^{1},...,S_{0}^{J}).
\]

In our model $G$ is given by \eqref{eqG function}. Clearly for $G$ of the form
\[
G (X,S^1,\cdots, S^J)=X-f(S^1,\cdots,S^J),
\]
\[
\mathbf{B}G(X,S^{1},...,S^{J})
 =X - \frac{1}{\rho}\min_{\ga}\max_{\xi}
 [f(\xi ^1 S^1,\xi^2 S^2, \cdots,\xi^J S^J)
-\sum_{j=1}^J\ga^j S^j(\xi^j-\rho)],
\]
and hence
\[
\mathbf{B}^nG(X,S^1,\cdots,S^J)
= X -\frac{1}{\rho^n}(\BC^{n}f)(S^1, \cdots ,S^J),
\]
where the {\it reduced Bellman operator} is defined as:
\begin{equation}
\label{eqBellmanforop}
(\BC f)(z^1,...,z^J)=\min_{\ga}\max_{\{\xi^j \in [d_j,u_j]\}}
[f(\xi ^1 z^1,\xi^2 z^2, \cdots,\xi^J z^J)
-\sum_{j=1}^J\ga^j z^j(\xi^j-\rho)],
\end{equation}
or, more concisely,
\begin{equation}
\label{eqBellmanforop2}
(\BC f)(z)=\min_{\ga}\max_{\{\xi^j \in [d_j,u_j]\}}
[f(\xi \circ z)-(\ga, \xi \circ z-\rho z)].
\end{equation}

This leads to the following result from \cite{Ko98}.

\begin{theorem}
\label{thoptionprice}
 The minimal value of $X_{0}$ for which the
income of the investor is non-negative (and which by definition is
the hedge price $H^n$ in the $n$-step game) is given by
\begin{equation}
H^{n}=\frac{1}{\rho^{n}}(\BC^{n}f)(S_0^1, \cdots , S_0^J).
 \label{hedgeprice}
\end{equation}
\end{theorem}

We shall now develop a method for evaluating the operator \eqref{eqBellmanforop} showing in particular how naturally the risk-neutral probability laws appear, as if by miracle, in this evaluation.  The proof of all results below (and its various extensions) can be found in \cite{Ko11}.

Let us say that a probability law $\mu \in \PC(E)$ on a compact set $E\subset \R^d$ is {\it risk-neutral with respect to the origin}, or shortly, {\it risk-neutral}\index{risk-neutral probability}
if the origin is its barycenter, that is $\int_E \xi \mu (d\xi)=0$.
The set of all risk-neutral laws on $E$ will be denoted by $\PC_{rn}(E)$.

 Let us say that a finite family of vectors $E=\{ \xi_1,\cdots,\xi_k\}$ in $\R^d$ is in
{\it general position}\index{general position} if the vectors of any subset of $\{ \xi_1,\cdots,\xi_k\}$ of size $d$ are linearly independent (in particular, all vectors in $E$ are non-vanishing).
A subset $E\subset \R^d$ is called (strongly)
{\it positively complete}, if
 there exists no $\om \in \R^d$ such that $(\om, \xi)\ge 0$ for all $\xi \in E$.

  \begin{prop}
\label{propbasriskneut2}
Let a finite set $E=\{ \xi_1,\cdots,\xi_{d+1}\}$ be strongly positively complete in $\R^d$. Then
the family $E$ is in general position and
there exists a unique risk-neutral probability law $\{p_1,\cdots , p_{d+1}\}$ on $\{ \xi_1,\cdots,\xi_{d+1}\}$,
with
 \begin{equation}
\label{eq1propbasriskneut2}
 p_i= C^{-1} (-1)^{i-1} \det \left(
 \begin{aligned}
& \xi_1^1 \quad \cdots \quad  \xi_{i-1}^1 \quad  \xi_{i+1}^1 \quad \cdots \quad \xi^1_{d+1}  \\
& \xi_1^2 \quad \cdots \quad  \xi_{i-1}^2 \quad  \xi_{i+1}^2 \quad \cdots \quad \xi^2_{d+1}  \\
&  \quad \quad \quad  \quad  \cdots \quad   \\
& \xi_1^d \quad \cdots \quad  \xi_{i-1}^d \quad  \xi_{i+1}^d \quad \cdots \quad \xi^d_{d+1}  \\
\end{aligned}
\right), \quad i=1,\cdots , d,
\end{equation}
where
 \begin{equation}
\label{eq2propbasriskneut2}
 C= \det \left(
 \begin{aligned}
& 1 \quad \quad 1 \quad \cdots \quad  1  \\
& \xi_1^1 \quad \xi_2^1 \quad \cdots  \quad \xi^1_{d+1}  \\
&  \quad \quad \quad  \quad  \cdots \quad   \\
& \xi_1^d \quad \xi_2^d \quad \cdots \quad \xi^d_{d+1}  \\
\end{aligned}
\right).
\end{equation}
 \end{prop}

For arbitrary $k$ we have the following.

\begin{prop}
\label{propextremriskneutr}
Let a family $E=\{\xi_1, \cdots, \xi_k\}$ be strongly positively complete and
in general position. Then
the extreme points of the convex set of risk-neutral probabilities on $\{\xi_1, \cdots, \xi_k\}$
 are risk-neutral probabilities with supports on strongly positively complete subsets of $E$
 of size precisely $d+1$.
\end{prop}

Now we are going to evaluate the minimax expression
\begin{equation}
\label{eq10mainminmaxoparbd}
\Pi[\xi_1, \cdots , \xi_k](f)=\inf_{\ga \in \R^d} \max_i [f(\xi_i)-(\xi_i,\ga)].
\end{equation}

 \begin{theorem}
\label{thriskneutrald}
Let a family of vectors $\xi_1,\cdots, \xi_k$ in $\R^d$
be strongly positively complete and in general position.
Then
\begin{equation}
\label{eq11mainminmaxoparbd}
\Pi[\xi_1, \cdots , \xi_k](f)=\max_{\mu} \E_{\mu}f(\xi),
\end{equation}
where $\max$ is taken over all extreme points $\mu$ of risk-neutral
laws on $\{ \xi_1,\cdots, \xi_k\}$,
given by Proposition \ref{propextremriskneutr}, and $\inf$ in \eqref{eq10mainminmaxoparbd}
is attained on a certain finite $\ga$.
\end{theorem}

Changing variables $\xi=(\xi^1,\dots, \xi^J)$ to
$\eta=\xi \circ z$ yields
\begin{equation}
\label{eqBellmanforop1}
(\BC f)(z^1,...,z^J)=\frac{1}{\rho}\min_{\ga}\max_{\{\eta \in [z^id_i,z^i u_i]\}}
[f(\eta)
-\sum_{i=1}^J\ga^i (\eta^i-\rho z^i)],
\end{equation}
or, by shifting,
\begin{equation}
\label{eqBellmanforop21}
(\BC f)(z^1,...,z^J)=\frac{1}{\rho} \min_{\ga}\max_{\{\eta \in [z^i(d_i-\rho),z^i (u_i-\rho)]\}}
[f(\eta + \rho z)-(\ga, \eta)].
\end{equation}
Assuming $f$ is convex (which is often the case for option payoffs),
we can apply Theorem \ref{thriskneutrald}, where $\max$ is taken over the set of vectors
\[
\eta_I= \xi_I \circ z -\rho z,
\]
being the vertices of the rectangular parallelepiped
\[
\Pi_{z,\rho}= \times_{i=1}^J [z^i(d_i-\rho),z^i (u_i-\rho)],
\]
where
\[
\xi_I
= \{ d_i|_{i\in I}, u_j|_{j \notin I}\},
\]
are the vertices of
\begin{equation}
\label{eqparalforop}
\Pi= \times_{i=1}^J [d_i, u_i],
\end{equation}
parametrized by all subsets (including the empty one) $I\subset \{1,\dots,J\}$.

Since the origin is an internal point of $\Pi$ (because $d_i<\rho <u_i$), the family $\{\eta_I\}$
is strongly positively complete. The condition of general position is rough in the sense that
 it is fulfilled for an open dense subset of pairs $(d_i,u_i)$. Applying
 Theorem \ref{thriskneutrald} to \eqref{eqBellmanforop21}
 and returning back to $\xi$ yields the following.

\begin{theorem}
\label{thbasicEuroprainbowBel}
If the vertices $\xi_I$ of the parallelepiped $\Pi$ are in general position in the sense that
for any $J$ subsets $I_1,\cdots, I_J$, the vectors $\{\xi_{I_k}-\rho \1\}_{k=1}^J$
are independent in $\R^J$, then
\begin{equation}
\label{eq1thbasicEuroprainbowBel}
(\BC f)(z)=\frac{1}{\rho} \max_{\{\Om\}} \E_{\Om} f(\xi \circ z), \quad z=(z^1,\cdots, z^J),
\end{equation}
where $\{\Om\}$ is the collection of all subsets $\Om =\xi_{I_1},\cdots,\xi_{I_{J+1}}$
of the set of vertices of $\Pi$, of size $J+1$, such that their convex hull
contains $\rho \1$ as an interior point ($\1$ is the vector with all coordinates 1), and where $\E_{\Om}$ denotes the expectation with respect to the
unique probability law $\{p_I\}$, $\xi_I\in \Om$,
 on the set of vertices of $\Pi$, which is supported on $\Om$
 and is risk-neutral with respect to $\rho \1$, that is
\begin{equation}
\label{eq2thbasicEuroprainbowBel}
\sum_{I\subset \{1,\dots,J\}} p_I\xi_I=\rho \1.
\end{equation}
\end{theorem}

Risk-neutrality now corresponds to its usual meaning in finance, i.e.
\eqref{eq2thbasicEuroprainbowBel} means that all discounted
stock prices are martingales.

Notice that the $\max$ in \eqref{eq1thbasicEuroprainbowBel} is over a finite
number of explicit expressions, which is of course a great achievement as compared
with initial minmax over an infinite set. In particular, it reduces the calculation
of the iterations $\BC^n f$ to the calculation on a controlled Markov chain.
Let us also stress that the number of eligible $\Om$ in \eqref{eq1thbasicEuroprainbowBel} is
the number of different pyramids (convex polyhedrons with $J+1$ vertices) with vertices
taken from the vertices of $\Pi$ and containing $\rho \1$ as an interior point. Hence this number can be effectively
calculated.

Let us point our some properties of the operator $\BC$ given by
\eqref{eq1thbasicEuroprainbowBel} that are obvious, but important for practical calculations:
 $\rho \BC$ is {\it non-expansive}:
\[
\| \BC (f_1)- \BC (f_2) \| \le \frac{1}{\rho} \| f_1-f_2\|,
\]
and {\it homogeneous} (both with respect to addition and multiplication):
\[
\rho \BC (\la + f)=\la + \rho \BC (f), \quad \BC (\la f)=\la \BC (f)
 \]
for any function $f$ and $\la \in \R$ (resp. $\la >0$) for the first (resp second) equation.

Next, if $f_p$ is a power function, that is
\[
f_p(z)= (z^1)^{i_1} \cdots (z^J)^{i_J},
\]
then $f_p(\xi \circ z)= f_p(\xi)f_p(z)$, implying
\begin{equation}
\label{eqBelonpowers}
(\BC ^n f_p)(z)=((\BC f_p)(\1))^n f_p(z).
\end{equation}
Therefore, power functions are invariant under $\BC$ (up to a multiplication by a constant).
Consequently, if for a payoff $f$ one can find a reasonable approximation by a power function,
that is there exists a power function $f_p$ such that
$\|f-f_p\| \le \ep$, then
\begin{equation}
\label{eqBelonpowersap}
\|\BC ^n f- \la ^n f_p \|\le \frac{1}{\rho ^n} \|f-f_p\| \le \frac{\ep}{\rho^n}, \quad \la = (\BC f_p) (\1),
\end{equation}
so that an approximate calculation of $\BC ^n f$ is reduced to the calculation of one number
$\la $. This implies the following scheme for an approximate evaluation of $\BC$: first
find the best fit to $f$ in terms of functions $\al +f_p$ (where $f_p$ is a power function and $\al$
a constant), and then use \eqref{eqBelonpowersap}.

\end{document}